\documentclass[a4paper,12pt]{article}
\usepackage{amssymb,amsmath,amsthm}
\usepackage{amsfonts}
\usepackage{mathrsfs}
\usepackage{latexsym}
\usepackage{multicol}
\usepackage{graphicx} 

\def\be{\begin{equation}}
\def\ee{\end{equation}}

\def\d'{``}

\newtheorem{thm}{Theorem}[section]

\newtheorem{propn}[thm]{Proposition}
\newtheorem{lem}[thm]{Lemma}

\newtheorem{conj}[thm]{Conjecture}

\theoremstyle{definition}

\newtheorem{rem}[thm]{Remark}

\def\be{\begin{equation}}
\def\ee{\end{equation}}
\def\bea{\begin{eqnarray}}
\def\eea{\end{eqnarray}}

\def\i'{\textrm{i}}

\def\d'{``}

\begin{document}

\newcommand{\beq}{\begin{equation}}  
\newcommand{\eeq}{\end{equation}}  
\newcommand\la{{\lambda}}   
\newcommand\La{{\Lambda}}   
\newcommand\ka{{\kappa}}   
\newcommand\al{{\alpha}}   
\newcommand\bet{{\beta}} 
\newcommand\gam{{\gamma}}     
\newcommand\om{{\omega}}  
\newcommand\tal{{\tilde{\alpha}}}  
\newcommand\tbe{{\tilde{\beta}}}   
\newcommand\tla{{\tilde{\lambda}}}  
\newcommand\tmu{{\tilde{\mu}}}  
\newcommand\si{{\sigma}}  
\newcommand\lax{{\bf L}}    
\newcommand\mma{{\bf M}}    
\newcommand\rd{{\mathrm{d}}}  
\newcommand\tJ{{\tilde{J}}}  
\newcommand\ri{{\mathrm{i}}} 

\newcommand{\N}{{\mathbb N}}
\newcommand{\Q}{{\mathbb Q}}
\newcommand{\Z}{{\mathbb Z}}
\newcommand{\C}{{\mathbb C}}
\newcommand{\R}{{\mathbb R}}

\title{A Hirota bilinear equation for Painlev\'e transcendents $P_{IV}$, $P_{II}$ and $P_I$.}
\author{A.N.W. Hone\thanks{SMSAS, University of Kent, Canterbury, U.K.}, 
and F. Zullo\thanks{Dipartimento di Ingegneria Meccanica e Aerospaziale, Universit\`a La Sapienza, Roma, Italy.}}

\maketitle

\begin{abstract}
\noindent
We present some observations on the tau-function for the fourth Painlev\'e equation. By considering a Hirota bilinear equation of order four for this tau-function, we describe the general form of the Taylor expansion around an arbitrary movable zero. The corresponding Taylor series for the tau-functions of the first and second Painlev\'e equations, 
as well as that for the 
Weierstrass sigma function, arise naturally as special cases, by setting certain parameters to zero.  
\end{abstract}

\section{Introduction} \label{intro}
The six Painlev\'e equations (denoted $P_{I}-P_{VI}$) can be considered as nonlinear analogues of the classical functions: they admit a Hamiltonian representation \cite{O}, all of them  (apart from $P_I$) possess B\"acklund transformations \cite{C}, and they each arise as a compatibility condition for an associated isomonodromy problem \cite{JM}. 
General solutions of  Painlev\'e equations have asymptotics  
in terms of elliptic functions, which  was originally obtained (for $P_I$ and $P_{II}$)  by Boutroux \cite{Boutroux}. It is also known that through a limiting procedure, usually called the coalescence cascade, it is possible to obtain all the equations $P_{V}-P_{I}$ just from equation $P_{VI}$ (see e.g.  \cite{Ince}). 
Furthermore, the equations $P_I$, $P_{II}$ and $P_{IV}$ share the property that all their local solutions are meromorphic and possess a meromorphic continuation in the whole complex plane \cite{HL}. 

The Hamiltonian functions for $P_I$-$P_{VI}$ are  polynomials $h_j=h_j(q,p,z)$ in the 
canonically conjugate phase space variables 
$q,p$, and are rational in  the independent variable $z$. Letting  a prime denote differentiation with respect to $z$, 
the Hamiltonian formulation allows each of the Painlev\'e equations to be formulated as a first order system, 
\beq\label{ham} 
q'=\frac{\partial h_j}{\partial p}, \qquad p'=-\frac{\partial h_j}{\partial q}, \qquad j=I,...,VI.
\eeq 

The functions $h_j$ themselves, as functions of the time $z$, solve certain differential equations; these functions, defined   by $\si_j(z)=h_j(q(z),p(z),z)$, where $q(z),p(z)$ satisfy (\ref{ham}), are usually called ``sigma functions'' \cite{JM, O}. Every solution of the  Painlev\'e equation can be written in terms of the solution of a corresponding 
differential equation for $\si_j$, which is of second order and second degree. Moreover, 
the sigma function is given in terms of the logarithmic derivative of a tau-function. In some sense, one can view the sigma function or the tau-function as being more fundamental than the solution of the Painlev\'e equation, since in applications 
(such as in the theory of random matrices \cite{fw}) these are usually the main objects of interest.   

In recent work \cite{HRZ} we have shown how the recursive formula for the coefficients in the Laurent series expansion of solutions of the first Painlev\'e equation can be considered as an extension of the analogous formula for the Weierstrass $\wp$ function. In addition, the recursive formulae for the Taylor expansion of the
tau-function around one of its zeros lead to natural extensions of the expressions  found by Weierstrass \cite{w} 
for the elliptic sigma function 
(not to be confused with the sigma function of the Painlev\'e equations). The key to these recursive formulae was the use of a Hirota bilinear equation for the tau-function, 
amenable to the same method that was applied to the elliptic sigma function in \cite{EE}.  

The purpose of this short article is to derive recursive formulae for the expansion of the tau-function of the fourth 
Painlev\'e equation around a movable zero. Bilinear equations for $P_{IV}$ tau-functions have been derived previously, either as a system of two equations relating two tau-functions \cite{hk}, or as a symmetric system involving three 
tau-functions (see e.g. Theorem 3.5 in \cite{noumi}). However, by starting from the equation for the sigma function $\si_{IV}$, 
we can use a single  Hirota bilinear equation of fourth order to obtain the Taylor series expansion of  
the  $P_{IV}$ tau-function 
around a zero. By exploiting the freedom in the definition of $\si_{IV}$, we introduce additional parameters into the 
sigma function equation, and show how the corresponding series solutions for both $P_{II}$ and $P_I$ arise directly from 
the same bilinear equation as degenerate special cases, by setting suitable parameters to zero,  while all of these series can be viewed as natural extensions of 
the elliptic case treated in  \cite{EE}.  

In the next section we briefly review the Hamiltonian formulation of the fourth Painlev\'e equation and  the 
corresponding sigma equation, before introducing 
a ``shifted''  sigma equation (given by (\ref{bsig}) below), which is suitable for studying series expansions around movable poles, as well as the degeneration to 
$P_{II}$, $P_I$ and elliptic functions. 
Section 3 is concerned with  the properties of the tau-function for $P_{IV}$, the corresponding bilinear equation, and the presentation of the main result, namely the recursion for the Taylor coefficients (Theorem \ref{main}). The fourth  
section is devoted to a numerical application of the main result, using it to calculate approximations to the zeros of a particular tau-function for $P_{II}$, and we end with some conclusions and suggestions for future work.   
 
\section{Hamiltonian and sigma equation for $P_{IV}$} 

The fourth Painlev\'e equation can be derived from the Hamiltonian function 
\beq\label{4ham} 
h_{IV}(q,p,z)=\zeta (qp^2-q^2p) +\zeta^{-1}\Big( (e_2-e_3) p +(e_3-e_1)q \Big) +(e_3-\zeta^2qp)z, 
\eeq  
where $\zeta\neq 0$ and $e_j$ for $j=1,2,3$ are parameters. 
The corresponding Hamilton's equations (\ref{ham}) are given explicitly by (hereafter a prime denotes a derivative with respect to $z$)
\beq\label{p4sys}
q' = \zeta q(2p-q) +\zeta^{-1}  (e_2-e_3)-\zeta^2 zq, \qquad 
p'= \zeta p(2q-p) +\zeta^{-1}  (e_1-e_3)+\zeta^2 zp. 
\eeq 
The 
ordinary differential equation of second order  satisfied by $q$ arises by eliminating $p$ from (\ref{p4sys}), to yield 
\beq\label{p4} 
q''=\frac{(q')^2}{2q}+\frac{3}{2}\zeta^2 q^3 +2\zeta^3 zq^2+\Big( \frac{1}{2}\zeta^4 z^2-\al \Big) q +
\frac{\bet}{q}, 
\eeq 
where
$$ \al= e_2+e_3-2e_1+\zeta^2, \qquad  \bet =-\frac{(e_2-e_3)^2}{2\zeta^2},   
$$ 
which (up to rescaling $q$ and $z$) is just the fourth 
 Painlev\'e equation $P_{IV}$. By symmetry, upon eliminating $q$ from (\ref{p4sys}),  it follows that $p$  satisfies 
\beq\label{altp4} 
p''=\frac{(p')^2}{2p}+\frac{3}{2}\zeta^2 p^3 -2\zeta^3 zp^2+\Big( \frac{1}{2}\zeta^4 z^2-\tilde{\al} \Big) p +
\frac{\tilde{\bet}}{p}, 
\eeq 
with 
$$ \tilde{\al}= e_1+e_3-2e_2-\zeta^2, \qquad  \tilde{\bet} =-\frac{(e_1-e_3)^2}{2\zeta^2},   
$$ 
so that $-p$ satisfies the same form (\ref{p4}) of $P_{IV}$ as $q$ does, but for different values of the parameters 
$\al, \bet$.  

There is a certain amount of redundancy in the choice of parameters used above. Although the parameter $\zeta$ appears inessential, as (providing it is non-zero) it can always be removed by rescaling $q,p$ and $z$,  it will be needed in what follows. As for the three quantities $e_j$, $j=1,2,3$,  the solutions of $P_{IV}$ only depend on the differences $e_j-e_k$, but the inclusion of the term $e_3z$ in (\ref{4ham}) shows that the sigma function 
$$ \si_{IV}(z)=h_{IV}(q(z),p(z),z)$$ also depends on the parameter 
\beq\label{mud} 
\mu^* =\frac{e_1+e_2+e_3}{3}. 
\eeq 
Indeed, by taking derivatives of the Hamiltonian with respect to $z$, it follows that the sigma function satisfies 
the following equation of second order and second degree: 
\beq\label{sig} 
(\si_{IV}'')^2-\zeta^4 (z\si_{IV}'-\si_{IV})^2 +4(\si_{IV}'-e_1)(\si_{IV}'-e_2)(\si_{IV}'-e_3)=0.  
\eeq   
Moreover, $q$ and $p$ are given in terms of the solution of the latter equation 
by 
\beq\label{pq} 
q=\frac{\si_{IV}''-\zeta^2  (z\si_{IV}'-\si_{IV})}{2\zeta(\si_{IV}'-e_1)}, 
\qquad p=\frac{\si_{IV}''+\zeta^2  (z\si_{IV}'-\si_{IV})}{2\zeta(\si_{IV}'-e_2)}.  
\eeq 
The freedom to permute $e_1,e_2,e_3$ shows that generically the same solution of (\ref{sig}) provides six 
different solutions of the equation (\ref{p4}), with different $\al$, $\bet$; this  is one manifestation of the affine $A_2$ symmetry for  $P_{IV}$ \cite{O}, which can be seen more easily from its   symmetric form \cite{noumi}. 

Henceforth we regard the sigma equation (\ref{sig}) as the fundamental object 
of interest, and proceed to consider the behaviour of solutions near singularities. Since $q(z)$ and $p(z)$ are both meromorphic for all $z\in\C$ (see e.g. \cite{HL} or \cite{Steinmetz1}), it follows from (\ref{4ham}) that 
$\si_{IV}(z)$ is also a globally meromorphic function, and it is straightforward to see that its only possible singularities are movable simple poles 
with a local  Laurent expansion of the form 
\beq\label{siglaur}
\si_{IV}(z)=\frac{1}{z-z_0}+B + O\Big((z-z_0)\Big), 
\eeq 
where both the pole position $z_0$ and the quantity $B$ (resonance parameter) are arbitrary. For fixed values of the coefficients $e_j$ and $\zeta$, any solution of the second order equation (\ref{sig}) is completely specified by a particular choice of the two values $z_0,B$ in (\ref{siglaur}), which is then determined on the whole complex plane by analytic continuation. 

In order to understand how the solution of (\ref{sig}) depends on the parameters $z_0,B$, it is convenient to shift 
\beq\label{trans}
z\to z+z_0, \qquad \si_{IV}\to \si +B, 
\eeq 
which leads to an equation of the form 
\begin{equation}\label{bsig}
\left(\Sigma '' \right)^2-\eta\left(z\Sigma '-\Sigma\right)^2
+2(\ka\Sigma '-\lambda)\left(z\Sigma '-\Sigma\right)+4(\Sigma')^3-g_2\Sigma '+g_3=0
\end{equation} 
where 
$$ 
\eta = \zeta^4 ,
$$ 
and  for $\mu=\mu^*+\eta z_0^2/12$, the dependent variable $\Sigma$ is given by 
\beq \label{bigsig}
\Sigma (z) =\si (z) -\mu z, 
\eeq 
with  
the parameters $\ka, \la, g_2,g_3$ being polynomials in $\eta,z_0,B$ and the $e_j$. 
Having fixed the pole to lie at $z=0$, and shifted away the parameter $B$, the function $\Sigma(z)$ satisfying (\ref{bsig})  
depends only on the 5 parameters $\eta,\ka,\la,g_2,g_3$, while $\sigma (z)$ depends on $\mu$ also. 
\begin{lem}\label{basic}
For $\zeta\neq 0$, via translations of the form (\ref{trans}),  there 
is a one-to-one correspondence between solutions of (\ref{sig}) 
with a pole at some $z_0\in\C$,  
and functions  
$$\sigma(z)=\Sigma(z) + \mu z$$ 
 with a pole at $z=0$, where $\Sigma (z)$ is the solution of  (\ref{bsig})  specified by the  local Laurent expansion 
\beq\label{bsigexp}
\Sigma(z)=\frac{1}{z}+ O(z^2). 
\eeq 
\end{lem} 

\begin{rem} The above result applies to any solution 
of (\ref{sig}) with at least one pole; in particular, this excludes certain trivial 
solutions which are linear in $z$. 
If we scale (\ref{p4}) so that $\zeta=1$, then 
all solutions of $P_{IV}$ which are transcendental, 
meaning  
that they are neither rational nor can be reduced to solutions of a Riccati equation,  have infinitely many simple poles with residue $+1$ and infinitely many with residue $-1$ \cite{GLS}.  The formula (\ref{pq}) shows that 
(for $\zeta=1$) $q$ has a pole with residue $-1$ at places where $\si_{IV}$ has a simple pole, and $q$ 
does not depend on the parameter $\mu$, so its behaviour near such a pole is completely determined by a function $\Sigma$ specified as above. 
Poles of $q$ with residue $+1$ correspond to places where $\si_{IV}$ has a zero with $\si_{IV}'\to e_1$; the  behaviour at such poles can also be determined by using the 
well known 
observation of Okamoto \cite{O} that when $\zeta=1$ every solution of (\ref{p4}) 
can be written as the difference of two Hamiltonians, i.e. 
$$ 
q(z) =\tilde{\si}_{IV}(z) -\si_{IV}(z),
$$ 
where  $\tilde{\si}_{IV}$ satisfies (\ref{sig}) but with suitably shifted parameters. 
\end{rem}

For future reference, we record the equation of third order that results by taking the derivative of (\ref{bsig}) and 
removing a factor of $\Sigma ''$, that is 
\beq\label{third} 
\Sigma'''+6(\Sigma')^2 - z(\eta z-2\ka )\Sigma'+(\eta z-\ka )\Sigma-\la z-\frac{1}{2}g_2=0. 
\eeq 
Clearly the parameter $g_3$ in (\ref{bsig}) is a first integral for the above equation. 

We now consider the  degenerate case $\eta=\zeta^4=0$, which  is no longer related to $P_{IV}$. 

\begin{propn} If $\eta=0$ and $\ka\neq 0$, then 
\beq
v=2\si '=2(\Sigma'+\mu) \qquad 
with
\qquad \mu = -  \la \ka^{-1} 
\eeq 
satisfies the $P_{XXXIV}$ equation in the form 
\beq\label{p34} 
vv''-\frac{1}{2}(v')^2+2v^3 +(\ka z -6\mu ) v^2 +\frac{\ell^2}{2}=0, 
\eeq 
where 
$$ 
\ell^2 = 16\mu^3-4g_2\mu -4g_3, 
$$ 
Thus 
\beq\label{p2sub}
u= \frac{v'+\ell}{2v} 
\eeq 
satisfies the second Painlev\'e equation $P_{II}$ in the form 
\beq\label{p2} 
u''=2u^3+(\ka z -6\mu )u+\ell-\frac{1}{2}\ka, 
\eeq
and conversely $v$ is given in terms of $u$ and its first derivative by 
\beq\label{p34sub}
v=-u'-u^2-\frac{1}{2}(\ka z -6\mu ). 
\eeq 
\end{propn}  
\begin{proof} If $\eta=0$, then (\ref{bsig}) reduces to the sigma equation for 
$P_{II}$, provided that $\ka\neq 0$.  Upon   multiplying (\ref{third}) by $v/2=\Sigma'-\la\ka^{-1}$ and 
subtracting off half of (\ref{bsig}), 
the terms involving $\Sigma$ are eliminated, and what remains is the equation (\ref{p34}) for $v$, which 
is referred to as $P_{XXXIV}$ in \cite{Ince}. Every solution of (\ref{p34}) gives a solution of (\ref{p2}), and vice-versa, 
according to the formulae  (\ref{p2sub}) and (\ref{p34sub}). 
\end{proof} 
\begin{rem} 
The relations  (\ref{p2sub}) and (\ref{p34sub}) can be rewritten as 
$$ 
v'=\frac{\partial h_{II}}{\partial u}, \quad 
u'=-\frac{\partial h_{II}}{\partial v}, \qquad \mathrm{with} 
\quad h_{II}=u^2v-\ell u+\frac{1}{2}v^2 + \frac{1}{2}(\ka z -6\mu )v,$$  
which is the Hamiltonian formulation of 
$P_{II}$ found in \cite{O}. 
The standard version  of (\ref{p34}), or that of (\ref{p2}), has $\mu=0$. However, 
the situation for $\eta=0$,  $\ka\neq 0$ is 
completely analogous to that in Lemma \ref{basic}: we can use expansions around $z=0$ for $\Sigma$, of the form,  
(\ref{bsigexp}) to obtain local Laurent expansions for the standard version of $P_{XXXIV}$ (or $P_{II}$)  
around a pole in an arbitrary position $z_0$.   
\end{rem}  

When $\eta=\ka=0$, then a further degeneration occurs. 

\begin{propn}  If $\eta=\ka =0$ and $\la\neq 0$, then 
$$ w= -\Sigma' $$ 
satisfies the first Painlev\'e equation $P_I$ in the form 
\beq\label{p1} 
w''=6w^2-\la z -\frac{1}{2}g_2, 
\eeq 
while if $\eta=\ka =\la= 0$, then the general solution of (\ref{bsig}) is given   in terms of the 
Weierstrass zeta function with invariants $g_2,g_3$ by 
\beq\label{zetafn} 
\Sigma (z) =\zeta (z-z_0;g_2,g_3)+B, 
\eeq 
for $z_0,B$ arbitrary, so $\Sigma'=-\wp (z-z_0; g_2,g_3)$ is an elliptic function of $z$.  
\end{propn} 
\begin{proof} Up to replacing $\la\to 6\la$, this coincides with the case considered in \cite{HRZ}.  
\end{proof} 

\section{Tau-function and bilinear equation}  

For the sigma equation in the form (\ref{bsig}), the tau-function
$\tau(z)$ is defined 
by 
\beq\label{taudef}
\Sigma (z) = \frac{d}{dz} \log \tau(z).  
\eeq  
Since $\Sigma$ is meromorphic, with its only singularities being simple poles with residue $+1$, the above formula implies that $\tau(z)$ is holomorphic, but is only defined 
up to overall scaling $\tau\to A\tau$ for an arbitrary non-zero constant $A$.  
By substituting (\ref{taudef}) into (\ref{bsig}), an equation of third order which is homogeneous of degree four in $\tau$ results, that is 
\be \label{quad} 
\begin{array}{ll}
\tau^2(\tau ''')^2 -6\tau\tau '\tau ''\tau ''' +4(\tau ')^3\tau ''' +4\tau (\tau '')^3 -3(\tau ' \tau '')^2  
&\\ -z(\eta z-2\ka)\Big(\tau\tau''-(\tau')^2\Big)^2 +2(\eta z-\ka)\Big(\tau^2\tau'\tau''-\tau(\tau')^3\Big) 
&\\+(2\la z -\eta +g_2)\tau^2(\tau')^2
-(2\la z +g_2)\tau^3\tau''+2\la\tau^3 \tau'+g_3\tau^4 &= 0.  
\end{array}
\ee  
Taylor expansions of (\ref{quad}) around a movable 
zero, which correspond to a movable simple pole in (\ref{bsig}),  take the form 
$$ 
\tau(z)=C_0(z-z_0)+C_1(z-z_0)^2+C_2(z-z_0)^3+\ldots, \qquad C_0\neq 0,  
$$ 
where $z_0$ (the position of the zero) and $C_0,C_1$ are arbitrary, while all subsequent coefficients are determined uniquely in terms of these three parameters. By considering gauge transformations of the form 
\beq\label{gauge} 
\tau(z)\to A\exp(Bz) \tau(z), \qquad A\neq 0,  
\eeq 
the initial coefficient $C_0$ can be set to 1, and $C_1$ can be set to $0$; in that case one can check that the 
next coefficient $C_2$ is also $0$. The overall effect of the transformation (\ref{gauge}) is to send 
$$ 
\Sigma \to \Sigma +B, 
$$
which results in changing the parameters in (\ref{bsig}) and 
(\ref{quad}). 
However, this change does not affect the form of the equation, and thus 
we obtain an alternative version 
of Lemma  \ref{basic}, reformulated in terms of the tau-function. 
\begin{lem}\label{basictau} 
For $\zeta\neq 0$, via translations of the form (\ref{trans}),  there 
is a one-to-one correspondence between solutions of (\ref{sig}) 
with a pole at some $z_0\in\C$,  
and functions  
$$\sigma(z)=\frac{d}{dz} \log \tau(z) +\mu z$$ 
 with a pole at $z=0$, where $\tau (z)$ is the solution of  (\ref{quad})  specified by the  local Taylor expansion 
\beq\label{bsigtay}
\tau(z)=z+ O(z^4). 
\eeq 
\end{lem} 

The degree four equation (\ref{quad}) is somewhat awkward for computing the coefficients in the local expansion 
\begin{equation}\label{tauseries}
\tau(z)=\sum_{n=0}^\infty C_nz^{n+1}
\end{equation}
around a zero at $z=0$. It is much more convenient to take the derivative of (\ref{quad}), so that after removing an overall factor  
one finds the bilinear (degree two) equation 
\be \label{bil} 
D_z^4 \tau\cdot \tau -z(\eta z-2\ka)D_z^2 \tau\cdot \tau +2(\eta z-\ka)\tau\tau'-(2\la z+g_2)\tau^2 =0, 
\ee 
which has been written concisely in terms of 
the Hirota derivative defined by 
\begin{equation}\label{hir}
D_z^n f\cdot g (z) = \left( \frac{d}{dz}-\frac{d}{dz'}\right)^n f(z)g(z') |_{z'=z}. 
\end{equation}
The equation (\ref{bil}) also follows immediately by making the substitution (\ref{taudef}) in (\ref{third}). The quantity 
$g_3$ in (\ref{quad}) also corresponds to a first integral of (\ref{bil}).


In order to describe the expansion of the tau-function around a zero, we 
use the bilinear equation (\ref{bil}), and note that  
the action of the Hirota operators $D_z^2$ and $D_z^4$ on monomials is 
given by   
$$ 
D_z^2 z^j\cdot z^k =a_{j,k}z^{j+k-2}, \quad D_z^4 z^j\cdot z^k =b_{j,k}z^{j+k-4}, 
$$
where the multipliers appearing on the right-hand side are 
$$
a_{j,k} =   2!\sum_{\ell =0}^2 (-1)^\ell \left(\begin{array}{c} j \\ \ell \end{array}\right) \, 
\left(\begin{array}{c} k \\ 2 - \ell \end{array}\right), \quad b_{j,k}=4!\sum_{\ell =0}^4 (-1)^\ell \left(\begin{array}{c} j \\ \ell \end{array}\right) \, 
\left(\begin{array}{c} k \\ 4 - \ell \end{array}\right). 
$$ 
The resulting recursion relation leaves the coefficients $C_0,C_1$ and $C_6$ undetermined. The freedom to chose $C_6$ in (\ref{bil}) corresponds to the value of the first integral $g_3$,   so in order to match 
the term at order $z^7$ arising from (\ref{quad}), the correct value of 
$C_6$ must be inserted in the recursion. 
Before stating the result, it is convenient to define 
the shifted multipliers 
$$ 
\hat{a}_{j,k}=a_{j,k}-2k, \qquad a^*_{jk}=a_{j,k}-k. 
$$ 
\begin{thm} \label{main} 
The coefficients $C_n$ in the expansion (\ref{tauseries}) obey the recursion
\begin{equation}\begin{array}{ll}
n(n^2-1)(n-6)C_n =&-\frac{1}{2}\sum_{j=1}^{n-1}b_{j+1,n+1-j}C_jC_{n-j} 
+\frac{1}{2}\eta\sum_{j=0}^{n-4}\hat{a}_{j+1,n-3-j}C_jC_{n-4-j}\\
&-\ka\sum_{j=0}^{n-3}a^*_{j+1,n-2-j}C_jC_{n-3-j}
+\frac{1}{2}g_2\sum_{j=0}^{n-4}C_jC_{n-4-j}\\ 
&+\lambda\sum_{j=0}^{n-5}C_jC_{n-5-j}. 
\end{array}\end{equation}
To obtain the expansion in the form  (\ref{bsigtay}), 
the free coefficients must be fixed as 
$$C_0=1, \qquad C_1=0, \qquad 
C_6=\frac{1}{5040}\ka^2-\frac{g_3}{840}. $$ 
With the latter choice, each coefficient $C_n$ is a weighted  
homogeneous polynomial  of
total degree $n$ in 
$\mathbb{Q}[\ka, \eta, g_2,\lambda ,g_3]$ 
with weights $3,4,4,5,6$ respectively, so that 
\beq \label{cform} 
C_n = \frac{\mathrm{P}_n (\ka, \eta, g_2,\lambda ,g_3)}{(n+1)!}, 
\eeq 
where $\mathrm{P}_n (\xi^3\ka, \xi^4\eta, \xi^4g_2,\xi^5\lambda ,\xi^6g_3)=\xi^n\mathrm{P}_n (\ka, \eta, g_2,\lambda ,g_3)$
for all $\xi\in\C^*$. 
\end{thm}

This above result extends the analogous recursion for the Taylor series 
coefficients of the Weierstrass sigma function \cite{EE} and for the 
tau-function of the first Painlev\'e equation \cite{HRZ}. 
We record the first few polynomials $\mathrm{P}_n$ here: 
$$ 
\begin{array}{ll}
\mathrm{P}_0=1, \qquad \,\,\mathrm{P}_1=\mathrm{P}_2=0, &  
\mathrm{P}_3=-\ka,\qquad  \mathrm{P}_4=2\eta-\frac{1}{2}g_2,\\ 
\mathrm{P}_5=-6\la, \quad \mathrm{P}_6=\ka^2-6g_3, &
\mathrm{P}_7 = -\ka (11\eta+g_2), \\
\mathrm{P}_8= 12\eta^2 +6g_2\eta+51\la\ka -\frac{9}{4}g_2^2, 
&
\mathrm{P}_9 = 17\ka^3 -42(\eta+g_2)\la +108 g_3\ka. 
\end{array} 
$$  
Computer calculations up to $\mathrm{P}_{100}$  suggest that, after suitable scaling of the variables 
$g_2,g_3$, these polynomials have integer coefficients. The form of the expression (\ref{cform}) implies that the Taylor series solution of (\ref{quad}) with leading order (\ref{bsigtay})  
can be written as  a multiple sum 
\be\label{tauWeiers} 
\tau(z) = \sum_{j,k,l,m,n\geq 0}A_{j,k,l,m,n} \, \ka^j\eta^k
\Big(\frac{g_2}{2}\Big)^l 
\lambda^m (6g_3)^n \frac{z^{3j+4k + 4l + 5m + 6n+1}}{(3j+4k + 4l + 5m + 6n+1)!}  
\ee 
where $A_{j,k,l,m,n} \in \Q$. 

\begin{conj} 
The series (\ref{tauWeiers}) has
$$
A_{j,k,l,m,n} \in \Z \quad \forall  j,k,l,m,n \geq 0. 
$$ 
\end{conj} 
\begin{rem}
In the case $\ka=\eta=\la=0$, Weierstrass \cite{w} considered the series 
 for the elliptic sigma function in the form (\ref{tauWeiers}), and 
Onishi proved that $2^l24^n A_{0 ,0,l,0,n}\in\Z$ for all $l,n$ \cite{onishi}, 
while in \cite{HRZ} we already found considerable numerical evidence to 
suggest that $A_{0,0,l,m,n}\in \Z$. 
\end{rem} 

The tau-function transforms in a very specific way when it is expanded around 
another zero, at a location $\Omega\neq 0$. 
\begin{propn} Let $\tau(z)=\tau (z; \eta,\ka , \la , g_2,g_3)$ denote the solution 
of (\ref{quad}) having the Taylor expansion  (\ref{bsigtay}) around $z=0$, and suppose that this function also vanishes at $z=\Omega\neq 0$. 
Then 
\beq\label{taddn} 
\tau (z+\Omega ; \eta,\ka , \la , g_2,g_3) 
=A \exp \Big(Bz +\frac{1}{2}\tilde{\mu}z^2\Big) 
	\tau (z; \eta,\tilde{\ka} , \tilde{\la} , \tilde{g}_2,\tilde{g}_3), 
\eeq 
where 
\beq\label{abmk} 
A=\tau'(\Omega), \qquad B = \frac{\tau ''(\Omega)}{2\tau'(\Omega)}, 
\qquad \tilde{\mu}=\frac{1}{12}\Omega(\Omega\eta-2\ka), 
\qquad 
\tilde{\ka}=\ka-\Omega\eta , 
\eeq 
$$ 
\tilde{\la} = \la-B\eta -\tilde{\ka}\tilde{\mu}, \qquad  \tilde{g}_2=g_2+12\tilde{\mu}^2+2\Omega\la+2B\tilde{\ka}, \qquad  
\tilde{g}_3=g_3-\tilde{g}_2\tilde{\mu}+4\tilde{\mu}^3-B^2\eta+2B\la. 
$$  
\end{propn} 
\begin{proof}
Upon replacing $z\to z+\Omega$ in (\ref{quad}), 
and introducing 
$$\tilde{\si}(z)=\frac{d}{dz}\log\tau (z+\Omega), $$ 
we see that $\tilde{\si}$ satisfies an  equation of the general form 
(\ref{sig}), and has a pole at $z=0$ because $\tau(\Omega)=0$.  
If we now set 
\beq\label{sigrel} 
\tilde{\si}(z) = 
\tilde{\Sigma}(z) +B+\tilde{\mu}z,  
\eeq
with $\tilde{\mu}$ given by the expression in (\ref{abmk}), 
then  for any  choice of $B$, $\tilde{\Sigma}(z) $ satisfies 
an equation of the canonical form (\ref{bsig}), but with different coefficients $\tilde {\ka}, \tilde{\la}, \tilde{g}_2, \tilde{g}_3$. Now we further require that 
$$ 
\tilde{\Sigma}(z) =\frac{d}{dz}\log \tilde{\tau}(z), 
$$ 
where $\tilde{\tau}$ has the Taylor expansion  (\ref{bsigtay}) around $z=0$. By integrating both sides of (\ref{sigrel}) and exponentiating, we 
see that 
$$
\tau (z+\Omega) = A \exp \Big(Bz +\frac{1}{2}\tilde{\mu}z^2\Big) 
\tilde{\tau}(z), 
$$ 
for some $A\neq 0$. 
By performing a Taylor expansion on each side of the above relation up to terms of order $z^3$ , we obtain the expressions for $A$ and $B$ as in 
  (\ref{abmk}), as well as the equation 
$$ 
\tilde{\mu}=\frac{\tau'''(\Omega)}{3\tau'(\Omega)} 
- \frac{\tau''(\Omega)^2}{4\tau'(\Omega)^2}. 
$$
The latter formula  
 can be seen to be consistent with the previous expression  for 
$\tilde{\mu}$ by setting $z=\Omega$ in (\ref{quad}).
Hence $\tilde{\tau}(z)$ satisfies  the same equation (\ref{quad}) 
but with parameters $\tilde {\ka}, \tilde{\la}, \tilde{g}_2, \tilde{g}_3$ 
found from the equation corresponding equation  (\ref{bsig}) for     $\tilde{\Sigma}(z)$. 
\end{proof} 
\begin{rem} 
The  expression (\ref{taddn}) is a generalization of the classical 
formula for transformation of the 
Weierstrass sigma function under shifting by a period (see e.g. $\S20.421$ in \cite{WW}). 
\end{rem} 

\section{Numerical example: poles in $P_{XXXIV}$}

\begin{figure} \centering
\includegraphics[width=10cm,height=10cm,keepaspectratio]{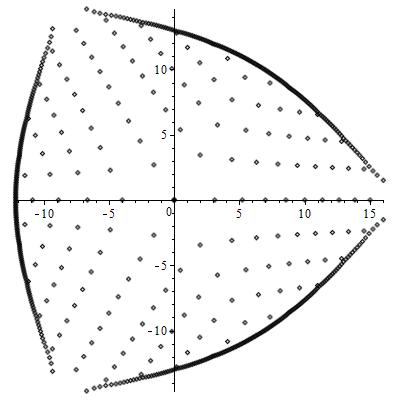}
\caption{\small{Approximation to the poles of the symmetrical solution of (\ref{sp34}).}}
\label{triangle}
\end{figure}
The numerical evaluation of Painlev\'e transcendents and the structure of their pole fields is a very active research area (see e.g. \cite{F1}-\cite{F3} and references therein). The recursion relation in Theorem \ref{main} is extremely convenient for 
computing numerical approximations to the tau-function close to the origin, by truncating the Taylor series for $\tau(z)$. The roots of the  polynomials obtained by truncation provide approximations to the zeros of the tau-function lying near to $z=0$, or equivalently the positions of the poles of the sigma equation. 

As an example, we consider the case of parameters 
\beq\label{sparams} \eta=\la=g_2=g_3=0, \qquad \ka =1, \eeq
for which the function 
$\tau(z)=\tau(z;0,1,0,0,0)$ is such that 
\beq\label{vtau} 
v(z)=\frac{d^2}{dz^2}\log\tau(z) 
\eeq 
satisfies the $P_{XXXIV}$ equation in the canonical form 
\beq\label{sp34} 
vv''-\frac{1}{2}(v')^2+2v^3 +z  v^2 =0, 
\eeq 
with parameter $\ell =0$, 
while from 
(\ref{p2sub}) we have that 
$$ 
u= \frac{1}{2} \frac{d}{dz}\log v 
$$ 
satisfies  $P_{II}$ in the form 
\beq\label{sp2} 
u''=2u^3+z u-\frac{1}{2}. 
\eeq

For the  parameter values (\ref{sparams}), the equation (\ref{bsig}) admits a trivial solution $\Sigma=\,$const, giving $v=0$, but we have neglected such solutions here, by considering the generic situation where $\Sigma$ has poles (cf. Lemma \ref{basic}). (In fact, setting $v=0$ in (\ref{p34sub}) 
yields a Riccati equation, corresponding to the special case that (\ref{sp2}) is solved in Airy functions.) 
The tau-function  $\tau(z)=\tau(z;0,1,0,0,0)$ given by 
the Taylor series defined in Theorem \ref{main} is such that the 
expansion 
(\ref{tauWeiers}) takes the special form 
\beq\label{cubes} 
\tau(z) = \sum_{j=0}^\infty \frac{\hat{A}_j \,z^{3j+1}}{(3j+1)!}, 
\qquad \hat{A}_j =A_{j,0,0,0,0}, 
\eeq  
which is invariant under the order 3 symmetry 
\beq\label{o3} 
z\to \om z, \qquad \tau\to \om^{-1}z, \qquad \om=\exp(2\pi \mathrm{i}/3). 
\eeq 
Hence the zeros of $\tau$ have the same symmetry: if $\Omega\neq 0$ is a zero of $\tau$, then so are $\om\Omega$ and $\om^2\Omega$. These 
zeros of $\tau$ are 
the simple poles of $\Sigma$, and the double poles of $v$, i.e. the particular solution of (\ref{sp34}) 
defined by (\ref{vtau}).  Similarly, the associated function $u$ that satisfies the case (\ref{sp2}) of $P_{II}$ has simple poles with residue $-1$ at these 
same positions, as well as simple poles with residue $+1$ at the places where $v$ vanishes, all of them symmetrically placed on triangles centred at $0$.   

For illustration, in Figure \ref{triangle} we have plotted the approximate positions of some of the zeros of $\tau$ (or the equivalently the poles of $\Sigma$ and $v$). To begin with, the first 201 non-zero terms of the series (\ref{cubes})  were found.  The 
first few coefficients have prime factorizations 
$$ \begin{array}{l} 
\hat{A}_0=1, \hat{A}_1=-1, \hat{A}_2=1,  \hat{A}_3=17,
\hat{A}_4=-557, \hat{A}_5=59\cdot 349,  \hat{A}_6= -1017719, \\  
\hat{A}_7 = 5\cdot7^2\cdot 59\cdot 4391, \hat{A}_8= -5\cdot 13\cdot 131\cdot 550439, 
\hat{A}_9=5^2\cdot 7 \cdot 2224640081, \\ 
\hat{A}_{10}=-5^2\cdot 570919\cdot 2406689, 
\hat{A}_{11}=5^2\cdot 41\cdot 61\cdot46043405509,  \ldots
\end{array} 
$$ 
and it appears to be the case that 
$$\hat{A}_j\equiv 0\pmod{5}  \quad \forall j\geq 7, \qquad 
 \hat{A}_j\equiv 0\pmod{7}  \quad \forall j\geq 14.
$$  
To prove either of these two statements seems not to be simple. However from the bilinear equation (\ref{bil}) it follows that the coefficients $\hat{A}_j$ are determined by the quadratic recurrence
\begin{equation*}\begin{array}{ll}
(n+1)\hat{A}_{n+3} =&\frac{9n^4+18n^3-85n^2-246n-144}{8}A_{n+2}+\\
&-\frac{1}{3}\sum_{j=0}^{n}\frac{(3n+7)!}{(3n-3j+4)!(3j+4)!}a_{3n-\frac{3}{2}j+6,\frac{9}{2}j+6}\hat{A}_{n-j+1}\hat{A}_{j+1}\\
&+\frac{1}{6}\sum_{j=1}^{n}\frac{(3n+7)!}{(3n-3j+7)!(3j+4)!}b_{3(n-j)+7,3j+4}\hat{A}_{n-j+2}\hat{A}_{j+1},
\end{array}\end{equation*}
subject to the initial conditions $\hat{A}_{0}=1$, $\hat{A}_{1}=-1$ and $\hat{A}_{2}=1$. The values of  $a_{j,k}$ and $b_{j,k}$ are defined by the action of the Hirota operators $D^{2}$ and $D^4$ on polynomials (see before Theorem \ref{main}). \\        
Given these coefficients, we took the polynomial 
$$ 
{\cal P}_{601}(z) =\sum_{j=0}^{200} \frac{\hat{A}_j \,z^{3j+1}}{(3j+1)!},
$$ 
and calculated its roots numerically in order to produce the figure. By comparing the values of the roots with those of the successive approximations ${\cal P}_{20}$,  ${\cal P}_{40}$, ${\cal P}_{60}$, ${\cal P}_{80}$,  etc. we were able to confirm that the values of the zeros    closest to $0$ were converging to a high degree of accuracy. For instance, the non-zero roots closest to the origin lie at $\Omega_1,\om\Omega_1,\om^2\Omega_1$, and the next closest roots are 
at   $\Omega_2,\om\Omega_2,\om^2\Omega_2$, 
where   
$$ 
\Omega_1\approx 3.10938452954168950042, 
\quad  
\Omega_2\approx-3.97992802289816587870,  
$$ 
to 20 decimal places. The largest roots of the polynomial, which can be seen to coalesce on the boundary of the figure, are 
numerical artefacts; they do not provide good approximations to the zeros of the tau-function.   

\begin{rem} Due to the homogeneity of the parameters $\ka$, $g_3$, all of the tau-functions $\tau (z;0,\ka , 0,0,g_3)$ admit the symmetry (\ref{o3}).   
\end{rem} 
\begin{rem} Non-polynomial rational solutions of the sigma equation are also included in the formulation of Lemma \ref{basic}. For example, 
 for the  parameter values 
$$ \eta=\la=g_2=0, \qquad \ka =1, \qquad g_3=-\frac{9}{16},  $$ 
the equation (\ref{bsig}) has the rational solution 
$$ 
\Sigma(z) = \frac{1}{z}-\frac{z^2}{8}, 
$$ 
corresponding to the tau-function 
$$ 
\tau(z; 0,1,0,0,-9/16)=z\exp(-z^3/24),  
$$ 
which             
illustrates the symmetry   (\ref{o3}) explicitly. 
This corresponds to 
$$ 
v=-\frac{2}{z^2}-\frac{z}{2}, \qquad u=-\frac{1}{z}  
$$ 
which are rational  solutions   of the $P_{XXXIV}$ equation (\ref{p34}) and the $P_{II}$ equation (\ref{p2}), respectively,  with 
$\ka =1$, $\mu=0$, $\ell=3/2$.   
\end{rem}

\section{Conclusions} 

Our analysis shows that the ``shifted'' sigma equation for $P_{IV}$, given by (\ref{bsig}), is a fundamental 
object which contains not only the general solution of $P_{IV}$,  but also that of $P_{II}$, $P_I$, and the Weierstrass 
$\wp$ function. Although the connection between Painlev\'e transcendents and elliptic functions 
 has a long history 
at the level of 
asymptotic expansions \cite{Boutroux}, the results presented here show that from the viewpoint of the sigma function 
the Painlev\'e transcendents are multi-parameter extensions of elliptic functions.  
Furthermore, although there is a coalescence cascade $P_{IV}\to P_{II}\to P_I$, this requires taking asymptotic limits of 
both the dependent and independent variables \cite{Ince}, whereas at the level of the solution of (\ref{bsig}) one has 
$ P_{IV}\supset P_{II}\supset P_I$, with the inclusion denoting that a parameter has been set to zero. It would be interesting to see if this approach can be extended to the sigma function of $P_{VI}$, in which case all the other Painlev\'e equations would be included as special cases. 

In future work we propose to consider Mittag-Leffler expansions of the solutions of (\ref{bsig}), and the asymptotic behaviour of the coefficients in Laurent expansions for the solutions of the sigma equation, as well as the corresponding Painlev\'e equations. It would also be good to obtain precise a priori bounds on the growth of the polynomials $\mathrm{P}_n$ appearing in Theorem \ref{main}, as this would yield an independent proof that the tau-function is holomorphic (hence  providing yet another proof of the Painlev\'e property for $P_I$, $P_{II}$ and $P_{IV}$; 
cf. \cite{Steinmetz1}). The arithmetic properties of 
the coefficients $A_{j,k,l,m,n}$ in (\ref{tauWeiers}) are also worthy of further study.

\noindent
{\bf Acknowledgements:} ANWH is supported by EPSRC fellowship  EP/M004333/1. FZ wishes to acknowledge the
financial support of the GNFM-INdAM, SISSA (Trieste) and CRM Ennio de Giorgi (Pisa) for participation in the workshop 
``Asymptotic and computational aspects of complex differential equations,'' held in Pisa 
from $13^{th}$-$17^{th}$ February 2017. Both authors are grateful to the organisers of the 
LMS-EPSRC Durham Symposium on Geometric and Algebraic Aspects of Integrability, $25^{th}$ July - $4^{th}$ August 2016, which gave us an opportunity to renew our collaboration.

\end{document}